\documentclass{amsart}
\usepackage[utf8]{inputenc}
\usepackage{tikz}
\usepackage{amssymb}
\usetikzlibrary{positioning, arrows, arrows.meta}
\usepackage{pgfplots}
\usepackage{chessfss}
\newtheorem{theorem}{Theorem}
\newtheorem*{theorem*}{Theorem}
\newtheorem{lemma}[theorem]{Lemma}
\newtheorem{claim}[theorem]{Claim}

\theoremstyle{definition}
\newtheorem{definition}[theorem]{Definition}
\newtheorem{remark}[theorem]{Remark}

\newcommand{\setof}[1]{\left\{#1\right\}}

\newcommand{\abs}[1]{\left\vert#1\right\vert}
\newcommand{\floor}[1]{\left\lfloor#1\right\rfloor}
\newcommand{\ceil}[1]{\left\lceil#1\right\rceil}
\newenvironment{subproof}[1][\proofname]{%
  \begin{proof}[#1]%
}{%
  \end{proof}%
}

\newcommand{\authorlogan}{
\author{Logan Crone}
\address{Logan Crone, University of North Texas, Department of Mathematics, 
1155 Union Circle \#311430, Denton, TX 76203-5017, USA}
\email{lcronest@gmail.com}
}

\newcommand{\authorcameron}{
\author{Cameron Bernstein}
\address{Cameron Bernstein, University of North Texas, Department of Mathematics, 1155 Union Circle \#311430, Denton, TX 76203-5017, USA}
\email{cameronbernstein@my.unt.edu}
}

\newcommand{\authorsydney}{
\author{Sydney Le}
\address{Sydney Le, University of North Texas, Department of Mathematics, 1155 Union Circle \#311430, Denton, TX 76203-5017, USA}
\email{sydneyle@my.unt.edu}
}

\newcommand{\authorlexie}{
\author{Alexandria Kwon}
\address{Alexandria Kwon, University of North Texas, Department of Mathematics, 1155 Union Circle \#311430, Denton, TX 76203-5017, USA}
\email{kwonlexie@gmail.com}
}


\title{Sneaky Angel and Devil Game}
\authorcameron
\authorlogan
\authorlexie
\authorsydney
\thanks{The authors would like to thank the Math Department of the University of North Texas's Incubator project for its encouragement and guidance.}
\begin{document}

\maketitle

\begin{abstract}
In this paper we introduce an imperfect information variant of the Angel and Devil game, which was first introduced in 1982 by Berlekamp, Conway, and Guy.  The Devil player has a winning strategy in this game, but the main problem they pose is whether this changes if the Angel player is allowed two moves for every one of the Devil.  This and many other variants of the game have been considered, and the original question was only solved in 2007, when it was shown that the so-called \emph{power 2} Angel could win.  Our main theorem, Theorem~\ref{mainthm}, is that the Devil player wins the imperfect information variant of the game. This generalizes the result that the Devil player has a winning strategy in the original Angel and Devil game.
\end{abstract}
\section{Introduction}
In mathematics, a game is characterized by the set of allowed moves and allowed positions that which produces sequences of moves resulting in wins, losses, or draws for each of the players.  Most games consist of two players alternating playing various moves. Games can be finite, where play ends once a terminal position has been reached, or can be infinite, with the players producing an infinite sequence of moves.  One can give a formal definition of a game in many different ways, and the formal details of the definition seldom affect the statements of theorems and proofs involved. The main property of interest in studying games is the existence of winning strategies.  A \emph{strategy} is a function which, given a finite sequence of moves played so far, will output the next move to play.  A \emph{winning strategy} is a strategy for one of the players that which will always result in a win for that player.  A game is then said to be \emph{determined} if one of the players has a winning strategy.

Games of \emph{perfect information} are those in which all players have full knowledge of the rules and payoff, as well as a complete list of all moves played by both players in each round.  In games of \emph{imperfect information}, some information is not revealed to the players.  In most games of imperfect information, it is the opponent's moves that are hidden.  Games with imperfect information are more likely to mirror interactions in reality, but are also more difficult to analyze, as the notion of a strategy must be adapted to this context as well.  The only type of imperfect information game we consider here is one in which one of the player's moves are hidden for some number of rounds from the other player.  For our purposes, a \emph{strategy} in this game will be playing deterministically (not probabilistically) against only the moves which have been revealed.  So for such a strategy to be a \emph{winning strategy}, it must win against all possible moves by the opponent.

In 1982, Berlekamp, Conway, and Guy \cite{BCG1982} introduced a game of perfect information in which one of the players is moving a piece on an infinite chess board, winning only if they are able to keep moving forever, while the other player opposes the first by damaging the board and attempting to trap their piece.  They posed a problem regarding which player has a winning strategy in the case that the piece is allowed to make two chess-king moves every round.  This problem stood unsolved for many years, with Conway \cite{Conway1996} revisiting the problem and Winkler \cite{Winkler2006} popularizing it.  Some progress was made by Kutz \cite{Kutz2004} in 2004, with the problem being fully answered in 2007 by Bowditch \cite{Bowditch2007} and M{\'a}th{\'e} \cite{Mathe2007}.

If a game of perfect information is sufficiently simple, then it is known that one of the players must have a winning strategy and is thus said to be determined, and the abstract study of which games are determined turned out to be a deep and rich theory.  This investigation into the determinacy of abstract games began in 1953 with Gale and Stewart \cite{GaleStewart1953}, culminating in 1975 with Martin \cite{Martin1975}.  The game introduced by Berlekamp, Conway, and Guy is simple enough to be determined by the theorem of Gale and Stewart, an equivalent statement of which we give here.

\begin{theorem*}[Gale and Stewart \cite{GaleStewart1953}]
Suppose $G$ is a game of perfect information in which one of the players wins all infinite plays and the other player wins all finite plays, then $G$ is determined.
\end{theorem*}

In Section \ref{S:defgame}, we first give the definition of the \emph{Angel and Devil game} originally introduced by Berlekamp, Conway, and Guy \cite{BCG1982}.  We then also define the imperfect information variant of the game, and it is for this variant we show that the Devil player has a winning strategy.

In Section \ref{S:mainresults}, we prove our results.  To do this we first consider restricted Angels which can only move upwards.  We define a sequence of strategies $\sigma_n$ which are able to capture these weakened Angels with increasing efficacy as $n \to \infty$ (this is made precise by Lemma~\ref{orange}).  It is this increasing efficacy that allows us to defeat Angels even with limited information.  We then adapt these strategies $\sigma_n$ to capture Angels that are not restricted, again showing that as $n \to \infty$, we can still win with less information about the Angel's moves.

Throughout this paper we use $\ceil{x}$ to denote the smallest integer which is greater than or equal to $x$, and $\floor{x}$ to denote the greatest integer which is less than or equal to $x$.

\section{Definition of the game}\label{S:defgame}
In this section, we give the formal definition of the original Angel and Devil game, as well as the variant on which we focus our attention on in this paper.

\begin{definition}[Angel and Devil game]
This is a game played on an infinite chess board, i.e. with \emph{squares} indexed by $\mathbb{Z}\times\mathbb{Z}$.  There are two players, the \emph{Angel player} and the \emph{Devil player}.  The Angel player controls an \emph{Angel piece} on the board, which moves like a chess king (see Figure~\ref{unresmove}).  For example if the Angel is on square $(0, 0)$, then the legal moves are to the squares
\[\{(-1, 1), (0, 1), (1, 1), (-1, 0), (1, 0), (-1, -1), (0, -1), (1, -1)\}\]

\begin{figure}
\begin{tikzpicture}
\def\gridlength{2}
\def\gridheight{2}
\draw [step=1 cm, draw=black!50!white, thin, fill=none] (-1,-1) grid (\gridlength, \gridheight);

\node (11) [inner sep=0.05cm] at (0.5, 0.5) {\boardfont k}; 
\coordinate (00) at (-0.5, -0.5);
\coordinate (01) at (-0.5, 0.5);
\coordinate (02) at (-0.5, 1.5);
\coordinate (10) at (0.5, -0.5);
\coordinate (12) at (0.5, 1.5);
\coordinate (20) at (1.5, -0.5);
\coordinate (21) at (1.5, 0.5);
\coordinate (22) at (1.5, 1.5);

\draw [->] (11) -- (00);
\draw [->] (11) -- (01);
\draw [->] (11) -- (02);
\draw [->] (11) -- (10);
\draw [->] (11) -- (12);
\draw [->] (11) -- (20);
\draw [->] (11) -- (21);
\draw [->] (11) -- (22);

\end{tikzpicture}
\caption{Legal moves for the Angel piece.}\label{unresmove}
\end{figure}
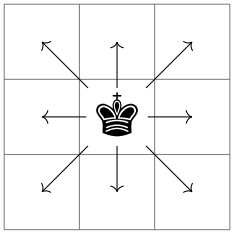

The Devil player, on each turn, \emph{deletes} a square from the board.  If a square is deleted, then the Angel player may no longer move to that square.
The Angel player is declared the winner if there is always some legal move for them to make.  The Devil player is the winner otherwise, i.e. the Devil player wins if they can eventually trap the Angel piece.  

Our convention is that the Angel piece starts on square $(0, 0)$, and that the Devil player \emph{may} delete the square which is occupied by the Angel, but the game only ends if the Angel has no legal move.  The Angel must move every turn and is \emph{not} allowed to pass.
\end{definition}

Next, we will define the variant of the Angel and Devil game to which our results apply. This variant is not a game of perfect information, and so we need to address what we mean by a \emph{strategy}, as well.

\begin{definition}[Sneaky Angel and Devil game]\label{sneakygame}
This is a game played exactly as the Angel and Devil game, but with a fixed parameter $s \in \setof{0, 1, 2, \dots}$.  We will say that the Angel is \emph{$s$-sneaky}, and the Devil only obtains in round $s+k$ the knowledge about the Angel piece's position on the board in round $k$.

For the sake of definiteness of how the game proceeds, we will have that the Angel piece starts on $(0,0)$ and the Angel player is allowed to make $s$ many moves in the first round without any interference from the Devil.  The Devil then deletes a square for the first time with no knowledge of the Angel's moves.  Then the Angel is allowed to move once more, and the Devil then gains the information about the Angel's first move, and is allowed to delete another square.  The game continues with each new move by the Angel revealing one more move to the Devil.

A \emph{strategy} for this game is just a function that takes the Devil's knowledge in round $k$ and decides (deterministically) a move for the Devil for that round.
\end{definition}

\section{Main Results}\label{S:mainresults}

In this section, we will first consider a restricted game in which the Angel can only move in one direction.  We will restrict the Angel player to always increase the $y$-coordinate of their Angel piece on every move. 
We define a sequence $\sigma_n$ of strategies that will guarantee capture of this upwards-only Angel with increasing effectiveness as $n$ increases (see Lemma~\ref{orange} for a more precise statement).  While it is relatively easy to compute by brute force that $\sigma_5$ will capture the original Angel, we will need $\sigma_n$ for arbitrarily large $n$ to capture the $s$-sneaky Angel where $s$ is large.

\begin{definition}
We define a sequence of strategies \emph{$\sigma_n$} which will prescribe moves based only on the Angel's last known position and the squares deleted so far.  We will be deleting only squares of the form $(a, n)$, with $y$-coordinate equal to $n$. If the Angel's last known position is $(x,k)$ after making $k$ many moves, we will have then deleted $k$-many squares on the row of squares with $y$-coordinate $n$. The strategy $\sigma_n$ is played in two stages: We will have one algorithm for the first $\ceil{n/2}$ many rounds, and another for the following $\floor{n/2}$ rounds.  We will define $\sigma_n$ to choose and delete the square satisfying the conditions we set for each stage.
For rounds $k+1 < \ceil{n/2}$, if we've deleted squares $\{(a_i, n)\}_{i \leq k}$ we will choose a square $(a_{k+1}, n)$ satisfying all of three of the following conditions:
\begin{enumerate}
\item $\abs{a_{k+1} - a_{i}} > 1$ for any $i \leq k$ 
\item $\abs{a_{k+1} - x} \leq \abs{b-x}$ for any $b$ so that $\abs{b - a_{i}}>1$ for every $i \leq k$
\item $a_{k+1} = \min\{b: b\ \text{satisfies Conditions (1) and (2)}\}$
\end{enumerate}

Condition (1) requires that we never play any square adjacent to one we've already deleted.  Condition (2) requires that we play as close to centered above the Angel's last known position as possible, and Condition (3) requires that we choose the leftmost square, if more than one square satisfies condition (1) and (2).

In order to define our strategy for later rounds, we need to do a simple computation when the Angel piece is halfway to row $n$.  We compute whether the Angel is currently on the left side or on the right side of the center of the collection of deleted squares.  More precisely, suppose we've deleted squares $\{(a_i, n)\}_{i\leq \ceil{n/2}}$, and the Angel is on square $(x, \ceil{n/2})$.  Let $c=\frac{1}{\ceil{n/2}} \sum^{\ceil{n/2}}_{i=1} a_i$.  If $x>c$, then we'll say \emph{the right side is light}, indicating that we have fewer deleted squares to the right of the Angel, otherwise we'll say that \emph{the left side is light}.

For rounds $k+1 \geq \ceil{n/2}$, if we've deleted squares $\{(a_i, n)\}_{i \leq k}$, and the Angel was last known to be at position $(x, k)$ we will choose a square $(a_{k+1}, n)$ satisfying both of the following conditions:
\begin{enumerate}
\item $\abs{a_{k+1} - x} \leq \abs{b-x}$ for any $b$ so that the square $(b, n)$ has not already been deleted. 
\item If there are two moves satisfying Condition (1), then we'll pick the one which is on the side which is light.
\end{enumerate}
In other words, we simply delete the square which is closest to the Angel piece's current $x$ coordinate, tending towards the side which had fewer squares deleted in round $\ceil{n/2}$, if there are two possible choices.
\end{definition}

We prove our main theorem by proving several lemmas about the strategies $\sigma_n$, and then adapting them to use against an unrestricted sneaky Angel.

\begin{theorem}\label{mainthm}
Let $s \in \mathbb{N}$.  The Devil has a winning strategy in the game in which the Angel is $s$-sneaky.
\end{theorem}

We first work towards proving that if $n$ is sufficiently large, then $\sigma_n$ will capture the sneaky Angel which can only move upwards.  We do this by proving a few lemmas about the squares which $\sigma_n$ deletes.

Our first lemma is regarding which squares will be deleted in the first half of the first $n$ rounds of the game.  We show that we are deleting only squares with even $x$-coordinates, and that we are, in fact, deleting a consecutive block of such squares.

\begin{lemma}\label{lemon}
Let $k\leq \ceil{\frac{n}{2}}$. Let $A$ be the set of $x$-coordinates of deleted squares after the Devil's $k^{\text{th}}$ move then
\[A = \{\min A + 2i : 0 \leq i \leq k-1\}.\]
\end{lemma}

\begin{proof}
We'll prove this by induction on $k$.
For the base case, the only square deleted is $(a_1, n)=(0, n)$, and so \[A=\{a_1\}=\left\{{\min A}\right\}=\left\{{\min A+2i\colon 0\leq i \leq 0}\right\}.\]

For the inductive step, assume that A = $\{\min A + 2i : 0 \leq i \leq k-1 \}$ is true for $k \geq 2$ where $k + 1 \leq \ceil{n/2}$. Suppose that the Angel was last seen on square $(x, k)$ (right before the Devil deletes the square $(a_{k+1}, n)$), and let $c = \frac{1}{2} (\min A + \max A)$.

Notice that the statement of Lemma~\ref{lemon} is equivalent to each next move $(a_{k+1}, n)$, chosen by $\sigma_n$, satisfying 
\[a_{k+1} = 
\begin{cases} 
\min A - 2 & x \leq c\\ 
\max A + 2 & x > c 
\end{cases}\]
To show that this is true, we prove the following claim.
\begin{claim}
$\min A - 1 \leq x \leq \max A + 1$.
\end{claim}
\begin{subproof}
Suppose for the sake of a contradiction that $x > \max A + 1$.  We consider which squares the Angel could have come from: one of $(x-1, k-1)$, $(x, k-1)$, or $(x+1, k-1)$ (see Figure~\ref{teleporting}).  It is impossible for the Angel to have been on $(x + 1, k - 1)$ or $(x, k - 1)$, since then we would have, by definition of $\sigma_n$, deleted the square $(a_k, n) = (\max A +2, n)$ or $(\max A + 3, n)$, since neither is ruled out by being adjacent to a previously deleted square.

\begin{figure} 
\begin{tikzpicture}
\def\gridlength{5}
\def\gridheight{3}
\draw [fill=black!45!white, draw=none] (0,2) rectangle (1, 3);
\draw [step=1 cm, draw=black!50!white, thin, fill=none] (-1,-1) grid (\gridlength, \gridheight);
\draw [white, fill=white] (-0.99, 1.15) rectangle (\gridlength-0.01, 1.85);

\node (32) at (2.5, 1.6) {\vdots}; 
\node (12) at (0.5, 1.6) {\vdots}; 
\node (32) at (1.5, 1.6) {\vdots}; 
\node (12) at (-0.5, 1.6) {\vdots}; 
\node (32) at (3.5, 1.6) {\vdots}; 
\node (12) at (4.5, 1.6) {\vdots}; 

\node (13) at (0.5, 2.5) {\tiny$\max A$}; 
\node (33) at (2.5, 2.5) {\tiny$x$}; 
\node (31) [inner sep=0.05cm] at (2.5, 0.5) {\boardfont k}; 

\draw [->] (1.5, -0.5) -- (31);
\draw [->] (2.5, -0.5) -- (31);
\draw [->] (3.5, -0.5) -- (31);

\end{tikzpicture}
\caption{The case $x > \max A + 1$}\label{teleporting}
\end{figure}
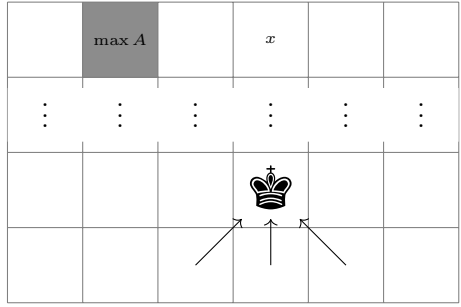

Therefore, the Angel must have come from $(x - 1, k - 1)$. However, since $x - 1 > \max A$, this also leads to a contradiction.  This is because $\max A$ must have been the last square deleted: by the induction hypothesis, the last square deleted must be either $\min A$ or $\max A$, but $\min A$ is clearly further away from $x$ in this case (unless $\min A = \max A$).  But since $\max A$ must have been the last square deleted, it was deleted while the Angel's last known position was $(x-1, k-1)$, but $x-1>\max A$ implies that $x-1$ would have been the square deleted by $\sigma_n$, as it was closer to the Angel's last known position. By a symmetric argument, $x < \min A - 1$ is also impossible.
\end{subproof}
Thus, following $\sigma_n$, either $(\max A + 2)$ or $(\min A - 2)$ is our next move as all other squares in the set $[\min A-1, \max A+1]$ are either adjacent or identical to previously deleted squares in $A$.
\end{proof}

Next, we show that the $\sigma_n$ always keeps the center of this block of consecutive even squares within one square of the Angel's last known position.

\begin{lemma}\label{lime}
Let $k \leq \ceil{n/2}$. Let $c$ be the average of the $x$-coordinates of the set of deleted squares $A$ after the Devil's $k^{\text{th}}$ move playing according to $\sigma_n$, i.e.
\[c = \frac{1}{2} (\min A + \max A)\]
Then if the Angel was last seen on the square $(x, k - 1)$ (i.e. before the Angel's $k^\text{th}$ move is revealed), then $\abs{c - x} \leq 1$.  Thus the Angel's current (hidden) $x$-coordinate is no more than $s+1$ many squares away from $c$.
\end{lemma}

\begin{proof}
We'll prove this also by induction on $k$.

After our first move $(a_1, n)=(0, n)$ we have that $c=0$, and that the Angel's last known position is the square $(x, 0) = (0, 0)$.  

\[\abs{c-x} = 0 \leq 1.\]

For the inductive step, suppose $0<k<\ceil{n/2}$ and the Angel was on square $(x_1, k-1)$ with the center $c_1$ satisfying $\abs{c_1-x_1} \leq 1$.
Suppose the Angel was seen moving to $(x_2, k)$.  By the rules of the game, we have that $\abs{x_2 - x_1} \leq 1$. 
Thus we know
\[\abs{c_1 - x_2} \leq\abs{c_1 - x_1} + \abs{x_1 - x_2} \leq 2.\]

Let $B = A \cup \{a_{k+1}\}$ where $\{a_{k+1}\}$ is our $(k+1)^\text{st}$ move, then by Lemma~\ref{lemon} we have
\[a_{k+1} =
\begin{cases} 
\max A + 2 & x_2 > c_1\\
\min A - 2 & x_2 \leq c_1
\end{cases}.
\]
Let $c_2 = \frac{1}{2}(\min B + \max B)$, then we have
\[c_2-c_1 =
\begin{cases} 
-1 & (c_1 - x_2) \geq 0\\
1 & (c_1 - x_2) < 0
\end{cases}.\]
Thus, sgn($c_2 - c_1$) = -sgn($c_1 - x_2$) and $\abs{c_2 - c_1} = 1$,
and therefore, 
\[\abs{c_2-x_2} = \abs{(c_1 - x_2) -(c_2 - c_1)} = \abs{c_1 - x_2} - \abs{c_2 - c_1} \leq 2 - 1 = 1.
\]
and so we maintain that $\abs{c_2 - x_2} \leq 1$ after our $(k+1)^\text{st}$ move.
\end{proof}

The next Lemma is our main result involving the analysis of these strategies $\sigma_n$ as $n$ gets large.

\begin{lemma}\label{orange}
Suppose the Devil has followed $\sigma_n$ until round $\ceil{n/2}+k$ of the game, with $2 \leq k \leq \floor{\floor{n/2}/2}$ and then Angel's last known position was the square $(x_1, \ceil{n/2} + k)$, then after following $\sigma_n$ for one more round and deleting the $\ceil{n/2} + k+1^{\text{st}}$ square, all of the squares in the interval $[x_1 - (k-2), x_1+(k-2)]$ have been deleted.
\end{lemma}
\begin{proof}
Suppose $A$ is the set of $\ceil{n/2}$-many squares that we deleted in the first half of the game, and that the Angel passed through the square $(x, \ceil{n/2})$, with $R=x+\floor{n/2}$ and $L=x-\floor{n/2}$. Let $B$ be the set of $(k+1)$-many squares deleted in the latter half of the game.  First note that $\abs{x_1 -x} \leq k$, and so 
\begin{align*}
R- (x_1 +(k-2)) 
& \geq R-(x+2k-2) 
\\ & \geq R-(x+2\floor{\floor{n/2}/2}-2) 
\\ & \geq R-(x+\floor{n/2}-2) 
\\ & = 2.
\end{align*}
In particular, we have that $R, R-1 \not \in [x_1 -(k-2), x_1+(k-2)]$, and likewise for $L, L+1$. So we have that the interval $[x_1 -(k-2), x_1+(k-2)]$ lies well-inside of the Angel's cone of reachable squares from the position $(x, \ceil{n/2})$.

We note now that by Lemma~\ref{lemon}, the length of $[\min A, \max A]$ is $2\ceil{n/2}-2$, whereas the length of $[L, R]$ is $2 \floor{n/2}$.  And also note that by Lemma~\ref{lime}, the intervals $[\min A, \max A]$, $[L, R]$ are centered within distance $1$ from one another.  So if two squares in $[L, R]$ are adjacent and both not in $A$, then one of them must be to the left of $L+2$ or or to the right of $R-2$.  Given that none of $L, L+1, R-1, R$ are in $[x_1 -(k-2), x_1+(k-2)]$, no pair of squares in $B\cap[x_1 -(k-2), x_1+(k-2)]$ is adjacent. And so by the same arguments as those in Lemmas~\ref{lemon} and \ref{lime} regarding the squares in $A$, we can conclude that the squares in $B \cap [x_1 -(k-2), x_1+(k-2)]$ must satisfy the conclusions of Lemmas~\ref{lemon} and \ref{lime} as well.

Now by applying Lemma~\ref{lime} to $B$, noting that $\max B - \min B = 2k$,
we have $\abs{x_1-(\max B-k)} \leq 1$.  In particular, we have
\[x_1 +k - 2 \leq \max B\]
and likewise 
\[\min B \leq x_1 -k + 2\]
and so $B$ contains every odd square in $[x_1 -(k-2), x_1+(k-2)]$, i.e. every square in this interval is either in $A$ or in $B$.
\end{proof}

\begin{figure} 
\begin{tikzpicture}[scale=0.15, every node/.style={scale=0.2}]
\draw [fill=black!45!white, draw=none] (-3,13) rectangle (16, 14);
\draw [fill=black!65!white, draw=none] (-10,13) rectangle (-9, 14);
\draw [fill=black!65!white, draw=none] (-8,13) rectangle (-7, 14);
\draw [fill=black!65!white, draw=none] (-6,13) rectangle (-5, 14);
\draw [fill=black!65!white, draw=none] (-4,13) rectangle (-3, 14);
\draw [fill=black!65!white, draw=none] (-2,13) rectangle (-1, 14);
\draw [fill=black!65!white, draw=none] (0,13) rectangle (1, 14);
\draw [fill=black!65!white, draw=none] (2,13) rectangle (3, 14);
\draw [fill=black!65!white, draw=none] (4,13) rectangle (5, 14);
\draw [fill=black!65!white, draw=none] (6,13) rectangle (7, 14);
\draw [fill=black!65!white, draw=none] (8,13) rectangle (9, 14);
\draw [fill=black!65!white, draw=none] (10,13) rectangle (11, 14);
\draw [fill=black!65!white, draw=none] (12,13) rectangle (13, 14);
\draw [fill=black!65!white, draw=none] (14,13) rectangle (15, 14);
\draw [fill=black!65!white, draw=none] (16,13) rectangle (17, 14);
\draw [fill=black!5!white, draw=none] (-2,12) rectangle (15, 13);
\draw [fill=black!5!white, draw=none] (-1,11) rectangle (14, 12);
\draw [fill=black!5!white, draw=none] (0,10) rectangle (13, 11);
\draw [fill=black!5!white, draw=none] (1,9) rectangle (12, 10);
\draw [fill=black!5!white, draw=none] (2,8) rectangle (11, 9);
\draw [fill=black!5!white, draw=none] (3,7) rectangle (10, 8);
\draw [fill=black!5!white, draw=none] (4,6) rectangle (9, 7);
\draw [fill=black!5!white, draw=none] (5,5) rectangle (8, 6);

\draw [step=1 cm, draw=black!50!white, thin] (-19,-25) grid (20, 15);
\node at (6.5,4.5) {\boardfont k}; 
\draw [black, ->] plot [smooth, tension=1] coordinates {(0.5, -24.5) (-1.5, -10.5) (4.5, -2.5) (6.5, 4.25)};
\end{tikzpicture}
\caption{The wall in place in round $\ceil{n/2}+\floor{\floor{n/2}/2}$.}\label{fig:3/4n}
\end{figure}

\begin{remark}\label{melon}
In any later round $\ceil{n/2}+k > \ceil{n/2}+\floor{\floor{n/2}/2}$, $\sigma_n$ is simply deleting the closest square to the Angel's last known $x$-coordinate which hasn't been deleted yet, and maintains that if the Angel was last seen on $(x_1, \ceil{n/2}+k)$, then all of the squares in $[x_1 - (\floor{\floor{n/2}/2}-2), x_1+(\floor{\floor{n/2}/2}-2)]$ have been deleted by the end of this round.  See Figure~\ref{fig:3/4n}.
\end{remark}

\begin{lemma}\label{pear}
For sufficiently large $n$, $\sigma_n$ is a winning strategy in the game where the Angel is $s$-sneaky, but must always increase the $y$-coordinate of the Angel piece.
\end{lemma}

\begin{proof}
By Lemma~\ref{orange} and Remark~\ref{melon}, if $n$ is large enough, i.e. if $\floor{\floor{n/2}/2}-2 \geq s$, then $\sigma_n$ will trap the sneaky Angel which is restricted to moving upwards.
\end{proof}

We will now consider another variant of the game in which the sneaky Angel is free to move also to the left or right, but still cannot decrease its $y$-coordinate. In other words, the Angel will be able to move ``side-to-side'' as in Figure~\ref{stsmove}.

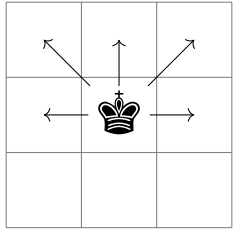
\begin{figure} 
\begin{tikzpicture}
\def\gridlength{2}
\def\gridheight{2}
\draw [step=1 cm, draw=black!50!white, thin, fill=none] (-1,-1) grid (\gridlength, \gridheight);

\node (11) [inner sep=0.05cm] at (0.5, 0.5) {\boardfont k}; 
\coordinate (01) at (-0.5, 0.5);
\coordinate (21) at (1.5, 0.5);
\coordinate (02) at (-0.5, 1.5);
\coordinate (12) at (0.5, 1.5);
\coordinate (22) at (1.5, 1.5);

\draw [->] (11) -- (01);
\draw [->] (11) -- (21);
\draw [->] (11) -- (02);
\draw [->] (11) -- (12);
\draw [->] (11) -- (22);

\end{tikzpicture}
\caption{Legal moves for side-to-side Angel}\label{stsmove}
\end{figure}
We will show that we can capture the sneaky Angel in this variant of the game, provided we have somehow already positioned walls out to the left and right of the Angel.  More formally, we will show that we can capture the Angel in the game where, for some large $m$, we begin the game with the squares of the form $(\pm m, k)$ for $0 \leq k \leq n$ already deleted, as in Figure~\ref{fig:stswalls}.  When we address the fully unrestricted sneaky Angel, we will need to begin play in such a way to simulate this head start.

\begin{figure} 
\begin{tikzpicture}[scale=0.2, every node/.style={scale=0.2}]
\draw [fill=black!65!white, draw=none] (-19,0) rectangle (-18, 10);
\draw [fill=black!65!white, draw=none] (19,0) rectangle (20, 10);

\draw [step=1 cm, draw=black!50!white, thin] (-19,0) grid (20, 15);
\node at (0.5,0.5) {\boardfont k}; 
\node [scale=4] at (-18.5,-0.75) {$-m$}; 
\node [scale=4] at (19.5,-0.75) {$m$}; 
\node [scale=4] at (-19.75,9.5) {$n$}; 

\end{tikzpicture}
\caption{Walls in place at the start of the side-to-side Angel game.}\label{fig:stswalls}
\end{figure}
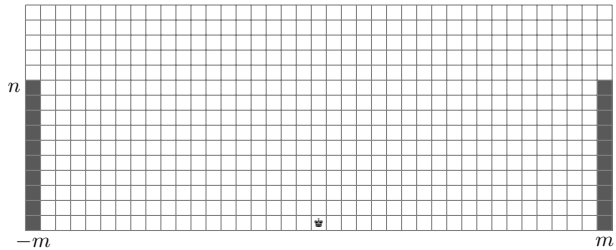

We define a variant of $\sigma_n$, which we will call $\hat \sigma_n$.  We will have $\hat \sigma_n$ play as though the Angel has moved upwards in every move, i.e. play $\sigma_n$'s response to the last known $x$-coordinate of the Angel, and ignore the actual $y$-coordinate.  Once the entirety of row $n$ between the walls has been deleted, $\hat\sigma_n$ will just fill in the rectangle $(-m, m) \times [0, n)$.

\begin{lemma}\label{apple}
For sufficiently large $n$, $\hat \sigma_n$ is a winning strategy in the game where the Angel is $s$-sneaky, but may never decrease the $y$-coordinate of the Angel piece.
\end{lemma}

\begin{proof}
By Lemma~\ref{orange} and Remark~\ref{melon}, after round $\ceil{n/2}+\floor{\floor{n/2}/2}$, we have an fully deleted an interval of squares on row $n$ of radius $\floor{\floor{n/2}/2}-2$ around the Angel's last known $x$-coordinate. The Angel certainly cannot escape upwards, provided $n$ is large enough so that $\floor{\floor{n/2}/2}-2 \geq s$.  By the assumption that we start the game with walls already in place to the left and right, the sneaky Angel can only exit the rectangle by moving upwards.  So eventually $\sigma_n$ will finish the top wall of the rectangle, and will do so early enough that the Angel couldn't have moved through the wall, and so is still trapped inside of the rectangle.
\end{proof}

We will now apply these lemmas to answer the question in the version of the game in which the Angel is free to move in any direction, and prove Theorem~\ref{mainthm}.

\begin{proof}[Proof of Theorem~\ref{mainthm}]
We will define a strategy $\Sigma_n$, which we will show captures the unrestricted $s$-sneaky Angel, provided $n$ is sufficiently large.

\begin{figure} 
\begin{tikzpicture}[scale=0.1, every node/.style={scale=0.1}]
\draw [fill=black!65!white, draw=none] (-29,23) rectangle (-22, 24);
\draw [fill=black!65!white, draw=none] (-22,23) rectangle (-23, 30);

\draw [fill=black!65!white, draw=none] (-29,-22) rectangle (-22, -23);
\draw [fill=black!65!white, draw=none] (-22,-22) rectangle (-23, -29);

\draw [fill=black!65!white, draw=none] (30,-22) rectangle (23, -23);
\draw [fill=black!65!white, draw=none] (23,-22) rectangle (24, -29);

\draw [fill=black!65!white, draw=none] (30,23) rectangle (23, 24);
\draw [fill=black!65!white, draw=none] (23,23) rectangle (24, 30);

\draw [fill=black!15!white, draw=none] (23, -22) rectangle (30,23); 
\draw [fill=black!15!white, draw=none] (23, -22) rectangle (-22,-29); 
\draw [fill=black!15!white, draw=none] (-22, -22) rectangle (-29,23); 
\draw [fill=black!15!white, draw=none] (-22, 23) rectangle (23,30); 

\draw [step=1 cm, draw=black!40!white, thin] (-29,-29) grid (30, 30);
\node at (0.5,0.5) {\boardfont k}; 

\end{tikzpicture}
\caption{How to trap an Angel.}\label{fig:trapbox}
\end{figure}
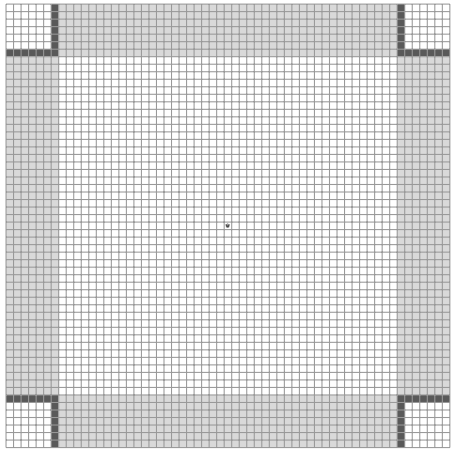

We define $\Sigma_n$ as follows:  First delete four corners of a large square box (see Figure~\ref{fig:trapbox}.  This will create four rectangular regions in the cardinal directions.  If the Angel attempts to escape through one of these regions, then we will switch strategies and appeal to Lemma~\ref{apple}.  More precisely, use the first $8n+4$ rounds of the game to delete the squares of the form $(x, y)$ where $x=a(9n+k)$ and $y=b(9n)$ or $x=a(9n)$ and $y=b(9n+k)$ where $a, b \in \setof{-1, 1}$ and $0 \leq k \leq n$, of which there are exactly $8n+4$.  Note that by the choice of these squares, the Angel must still be in the interior of the box $[-9n, 9n]^2$, since they have only moved $8n+4+s$ times so far, and if $\floor{\floor{n/2}/2}-2 \geq s$, and so $4+s \leq \floor{\floor{n/2}/2}+2 \leq n$, so that $8n+4+s\leq 9n$.

Now that we have these corners in place, we will have $\Sigma_n$ methodically delete all of the squares in $[-10n, 10n]^2$, so long as the Angel stays in the interior of the box $[-9n, 9n]^2$.  If the Angel's last known position ever moves out of $[-9n, 9n]^2$, it must be into one of the rectangular regions enclosed by the walls we've put into place.  Once the Angel enters one of these regions, we play essentially by the strategy $\hat\sigma_n$, again ignoring the Angel's forward progress (``forward'' depending on which of the four rectangular regions, e.g. upward, if the Angel is in the upper rectangle), and just playing in response to the Angel's (last known) relative side-to-side position inside of the rectangle. By Lemma~\ref{apple}, the Angel can't escape to the outside of the rectangular region.  If the Angel stays in the rectangle forever, then by following $\hat \sigma_n$, we will eventually finish the outer wall of the rectangle.  At this point, we switch to filling in the rectangular region containing the Angel's last known position, which must eventually drive the Angel back out of the rectangular region and into the box $[-9n, 9n]^2$.  If the Angel exits the rectangle before we finish the outer wall, or is forced out by the deletion of the entire rectangle, we switch back to filling in $[-10n, 10n]^2$ so long as the Angel's last known position is in $[-9n, 9n]^2$.  If the Angel visits (or revisits) another of the four rectangular regions, we do as prescribed above and follow $\hat \sigma_n$ to prevent the Angel from exiting to the outside of the rectangle.  We continue playing in this manner until the entirety of $[-10n, 10n]^2$ is filled in, at which point the Angel certainly has no legal moves and loses.

\end{proof}

\bibliographystyle{amsplain}
\bibliography{bibliography}

\end{document}